\documentclass[12pt]{article}
\usepackage{amssymb,amsthm,amsmath,amsfonts,latexsym,tikz,hyperref}
\usepackage[hmargin=1in,vmargin=1in]{geometry}

\usetikzlibrary{arrows.meta}
\usetikzlibrary{decorations.markings}
\tikzset{->-/.style={decoration={
  markings,
  mark=at position 0.6 with {\arrow{Latex[length=7pt]}}}, postaction={decorate}}}

\newtheorem{thm}{Theorem}[section]

\newtheorem{prop}[thm]{Proposition}
\newtheorem{cor}[thm]{Corollary}
\newtheorem{lem}[thm]{Lemma}
\newtheorem{conj}[thm]{Conjecture}
\newtheorem{exa}[thm]{Example}

\newcommand{\ora}{\overrightarrow}
\newcommand{\ola}{\overleftarrow}

\newcommand{\symmdiff}{\bigtriangleup}

\DeclareMathOperator{\Fix}{Fix}

\newcommand{\ben}{\begin{enumerate}}
\newcommand{\een}{\end{enumerate}}
\newcommand{\ble}{\begin{lem}}
\newcommand{\ele}{\end{lem}}
\newcommand{\bth}{\begin{thm}}
\renewcommand{\eth}{\end{thm}}
\newcommand{\bpr}{\begin{prop}}
\newcommand{\epr}{\end{prop}}
\newcommand{\bco}{\begin{cor}}
\newcommand{\eco}{\end{cor}}
\newcommand{\bcon}{\begin{conj}}
\newcommand{\econ}{\end{conj}}
\newcommand{\bde}{\begin{defn}}
\newcommand{\ede}{\end{defn}}
\newcommand{\bex}{\begin{exa}}
\newcommand{\eex}{\end{exa}}
\newcommand{\barr}{\begin{array}}
\newcommand{\earr}{\end{array}}
\newcommand{\btab}{\begin{tabular}}
\newcommand{\etab}{\end{tabular}}
\newcommand{\beq}{\begin{equation}}
\newcommand{\eeq}{\end{equation}}
\newcommand{\bea}{\begin{eqnarray*}}
\newcommand{\eea}{\end{eqnarray*}}
\newcommand{\bal}{\begin{align*}}
\newcommand{\bce}{\begin{center}}
\newcommand{\ece}{\end{center}}
\newcommand{\bpi}{\begin{picture}}
\newcommand{\epi}{\end{picture}}
\newcommand{\bpp}{\begin{picture}}
\newcommand{\epp}{\end{picture}}
\newcommand{\bfi}{\begin{figure} \begin{center}}
\newcommand{\efi}{\end{center} \end{figure}}
\newcommand{\bprf}{\begin{proof}}
\newcommand{\eprf}{\end{proof}\medskip}

\newcommand{\bsl}{\begin{slide}{}}
\newcommand{\esl}{\end{slide}}
\newcommand{\bfr}{\begin{frame}}
\newcommand{\efr}{\end{frame}}

\newcommand{\hso}[1]{\hspace{-1pt}}

\newcommand{\emp}{\emptyset}

\def\<{\langle}
\def\>{\rangle}

\newcommand{\ra}{\rightarrow}

\newcommand{\io}{\iota}
\newcommand{\ka}{\kappa}
\newcommand{\la}{\lambda}

\newcommand{\om}{\omega}

\newcommand{\bx}{{\bf x}}

\newcommand{\cM}{{\cal M}}

\newcommand{\cO}{{\cal O}}

\newcommand{\cS}{{\cal S}}

\DeclareMathOperator{\sgn}{sgn}

\DeclareMathOperator{\wt}{wt}

\begin{document}
\pagestyle{plain}
\title{Bijective proofs of proper coloring theorems
}
\author{
Bruce E. Sagan\\[-3pt]
\small Department of Mathematics, Michigan State University,\\[-5pt]
\small East Lansing, MI 48824, USA, {\tt sagan@math.msu.edu}\\[10pt]
Vincent Vatter\footnote{Vatter's research was partially supported by the Simons Foundation via award number 636113.}\\[-3pt]
\small Department of Mathematics, University of Florida,\\[-5pt]
\small Gainesville, FL  32601, USA, {\tt vatter@ufl.edu}
}

\date{\today\\[10pt]
	\begin{flushleft}
	\end{flushleft}}

\maketitle

\begin{abstract}
The chromatic polynomial and its generalization, the chromatic symmetric function, are two important graph invariants. Celebrated theorems of Birkhoff, Whitney, and Stanley show how both objects can be expressed in three different ways: as sums over all spanning subgraphs, as sums over spanning subgraphs with no broken circuits, and in terms of acyclic orientations with compatible colorings. We establish all six of these expressions bijectively.  In fact, we do this with only two bijections, as the proofs in the symmetric function setting are obtained using the same bijections as in the polynomial case and the bijection for broken circuits is just a restriction of the one for all spanning subgraphs.
\end{abstract}

\section{Introduction.}

Birkhoff first defined the chromatic polynomial in 1912 for planar graphs~\cite{bir:dfn} with the (ultimately unsuccessful) intention of proving what is now the four color theorem. Indeed, in a 1946 survey on the chromatic polynomial, Birkhoff and Lewis divided attacks on this then-conjecture into two types, qualitative (Type 1) and quantitive (Type 2), placing the study of chromatic polynomials at the center of the efforts of Type 2. Even then, the outlook was not rosy, as they admitted~\cite[p.~357]{bl:cp} (emphasis in original):
\begin{quote}
These researches have not been as successful as the researches of Type 1 in yielding results that are directly connected with the four-color problem. It is certain that the greater generality of the problem here considered has introduced complications which have \emph{so far} rendered the solution of the classical four-color problem more remote  by the methods characteristic of Type 2 than it is by the methods of Type 1.
\end{quote}

Despite its failure in leading to a proof of the four color theorem, the chromatic polynomial has nevertheless become a fundamental object in graph theory. We refer to Read's still-excellent 1968 survey~\cite{rea:icp} or the second part of Biggs's \emph{Algebraic Graph Theory}~\cite{big:agt} for more background on the chromatic polynomial and its importance. Our aim here is to explain, via simple bijections, three classic expressions for the chromatic polynomial---as a sum over spanning subgraphs, as a sum over spanning subgraphs without broken circuits, and in terms of acyclic orientations together with compatible colorings. In each case, we show how these bijections can also be used to prove analogous results for the chromatic symmetric function, a generalization of the chromatic polynomial defined by Stanley~\cite{sta:sfg} in 1995. First, though, we need definitions.

Let $G=(V,E)$ be a graph with vertices $V$ and edges $E$. A \emph{coloring} of $G$ is a function $\ka$ from $V$ to a set of \emph{colors}. Here our colors are exclusively the positive integers $\mathbb{P}$ and initial segments thereof, for which we use the notation $[t]=\{1,2,\dots,t\}$. Thus a \emph{$\mathbb{P}$-coloring} of $G$ is a function $\ka:V\to\mathbb{P}$ and a \emph{$[t]$-coloring} is a function $\ka:V\to[t]$. An example of a graph is shown on the left of Figure~\ref{gc} and two $[4]$-colorings of it are shown in the center.

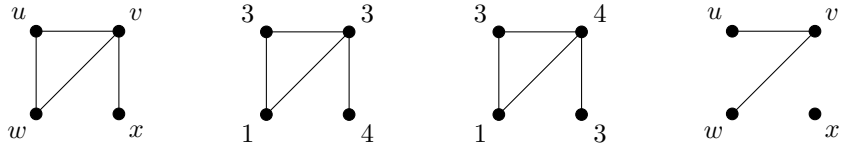
\begin{figure}
\begin{center}
\begin{footnotesize}
	\begin{tikzpicture}[scale=1.1]
	\draw (0,0)--(1,1)--(0,1)--(0,0) (1,0)--(1,1);
	\draw [fill=black] (0,1) circle (2pt) node [above left, circle, inner sep=4pt] {$u$};
	\draw [fill=black] (1,1) circle (2pt) node [above right, circle, inner sep=4pt] {$v$};
	\draw [fill=black] (0,0) circle (2pt) node [below left, circle, inner sep=4pt] {$w$};
	\draw [fill=black] (1,0) circle (2pt) node [below right, circle, inner sep=4pt] {$x$};
	\end{tikzpicture}
\qquad
	\begin{tikzpicture}[scale=1.1]
	\draw (0,0)--(1,1)--(0,1)--(0,0) (1,0)--(1,1);
	\draw [fill=black] (0,1) circle (2pt) node [above left, circle, inner sep=4pt] {$3$};
	\draw [fill=black] (1,1) circle (2pt) node [above right, circle, inner sep=4pt] {$3$};
	\draw [fill=black] (0,0) circle (2pt) node [below left, circle, inner sep=4pt] {$1$};
	\draw [fill=black] (1,0) circle (2pt) node [below right, circle, inner sep=4pt] {$4$};
	\end{tikzpicture}
\qquad
	\begin{tikzpicture}[scale=1.1]
	\draw (0,0)--(1,1)--(0,1)--(0,0) (1,0)--(1,1);
	\draw [fill=black] (0,1) circle (2pt) node [above left, circle, inner sep=4pt] {$3$};
	\draw [fill=black] (1,1) circle (2pt) node [above right, circle, inner sep=4pt] {$4$};
	\draw [fill=black] (0,0) circle (2pt) node [below left, circle, inner sep=4pt] {$1$};
	\draw [fill=black] (1,0) circle (2pt) node [below right, circle, inner sep=4pt] {$3$};
	\end{tikzpicture}
\qquad
	\begin{tikzpicture}[scale=1.1]
	\draw (0,0)--(1,1)--(0,1);
	\draw [fill=black] (0,1) circle (2pt) node [above left, circle, inner sep=4pt] {$u$};
	\draw [fill=black] (1,1) circle (2pt) node [above right, circle, inner sep=4pt] {$v$};
	\draw [fill=black] (0,0) circle (2pt) node [below left, circle, inner sep=4pt] {$w$};
	\draw [fill=black] (1,0) circle (2pt) node [below right, circle, inner sep=4pt] {$x$};
	\end{tikzpicture}
	\caption{A graph, two colorings, and a spanning subgraph.}
	\label{gc}
\end{footnotesize}
\end{center}
\end{figure}

The edge $e=uv\in E$ is \emph{monochromatic} in the coloring $\ka$ if $\ka(u)=\ka(v)$, and a coloring is \emph{proper} if none of its edges are monochromatic. Note that only the second coloring shown in Figure~\ref{gc} is proper. Indeed, every proper coloring of the graph shown on the left of Figure~\ref{gc} must give distinct colors to the vertices $u$, $v$, and $w$, but $x$ may be given either its own distinct color or it may share a color with $u$ or $w$. Thus the number of $[t]$-colorings of this graph is
\[
	t(t-1)^2(t-2),
\]
because there are $t(t-1)(t-2)$ ways to choose the colors of $u$, $v$, and $w$, and then there are $t-1$ choices for the color of $x$ (it may have any color except that given to $v$).

The quantity above is a polynomial in $t$. It turns out that this is always the case. In fact, this was the main result of Birkhoff's seminal paper~\cite{bir:dfn}, and he established it by proving the result we present as Theorem~\ref{PAllS}. Thus we define the \emph{chromatic polynomial} of the graph $G$ by
\[
	\chi(G;t)=\text{the number of proper $[t]$-colorings of $G$}
\]
for all positive integers $t$.

A \emph{spanning subgraph} of a graph $G=(V,E)$ is a graph on the same set of vertices, but with a subset of the edges. For example, the graph on the right in Figure~\ref{gc} is a spanning subgraph of the graph on the left.
We associate each subset $S\subseteq E$ of edges with the spanning subgraph of $G$ with edge set $S$, so the spanning subgraph on the right of Figure~\ref{gc} is denoted by $S=\{uv, vw\}$.
A \emph{connected component}, or simply \emph{component}, $K$ of a spanning subgraph $S$ is a maximal subset of vertices such that one can reach any vertex in $K$ from any other by traveling vertex-to-vertex via the edges of $S$.
The components of any graph (or spanning subgraph) form a partition of its vertices.
Returning to the spanning subgraph on the right of Figure~\ref{gc}, we see that it has two components, one with vertex set $\{u,v,w\}$ and the other consisting of the single vertex $x$.
We let
\[
	c(S) =\text{the number of components of $S$}.
\]

In our proofs we frequently consider certain \emph{improper} colorings. Given a spanning subgraph $S\subseteq E$ and a coloring $\ka$, we say that $\ka$ is \emph{monochromatic on the components of $S$} if $\ka(u)=\ka(v)$ for all vertices $u$ and $v$ in the same component of $S$.
These are precisely the colorings such that $\ka(u)=\ka(v)$ for all $uv\in S$.
If $\ka$ is a proper coloring, it follows that $\ka$ can only be monochromatic on the components of $S$ if $S=\emp$. The following lemma, while trivial to prove, is extremely useful in what follows.

\begin{lem}
\label{IM(S;[t])}
For every graph $G=(V,E)$ and spanning subgraph $S\subseteq E$, the number of $[t]$-colorings of $G$ that are monochromatic on the components of $S$ is $t^{c(S)}$.
\end{lem}
\begin{proof}
Every such coloring is obtained by independently choosing one of the $t$ available colors for each component of $S$.
\end{proof}

Since some our results involve signs, it is natural to use sign-reversing involutions to prove them. Let $\cS$ be a finite set with an associated sign function $\sgn:\cS\to\{+1,-1\}$. A \emph{sign-reversing involution} on $\cS$ is an involution $\io:\cS\ra\cS$ that satisfies the following two conditions.
\begin{itemize}
	\item If $\io(s)=s$, then $\sgn(s) = +1$.
	\item If $\io(s)=t$ and $s\neq t$, then $\sgn(s) = -\sgn(t)$.
\end{itemize}
Immediately from the definition we see that
\[
	\sum_{s\in\cS}\sgn(s)
	=
	|\Fix\io|,
\]
where $\Fix\io=\{s\in\cS \mid \io(s)=s\}$ is the set of fixed points of $\io$, and the vertical bars denote cardinality.

In order to generalize the chromatic polynomial to the chromatic symmetric function, let $\bx=\{x_1,x_2,\dots\}$ be an infinite set of commuting variables indexed by the positive integers. To every $\mathbb{P}$-coloring of the graph $G=(V,E)$ we associate the monomial or \emph{weight}
\[
	\bx^\ka=\prod_{v\in V} x_{\ka(v)}.
\]
For example, both colorings in Figure~\ref{gc} have monomial $\bx^\ka=x_1 x_3 x_3 x_4 = x_1 x_3^2 x_4$. The
\emph{chromatic symmetric function} of $G$ is then
\[
	X(G;\bx)=\sum_{\ka} \bx^\ka,
\]
where the sum is over all proper $\mathbb{P}$-colorings of $G$. That $X(G;\bx)$ is a symmetric function follows from the observation that permuting the labels of the colors in a proper coloring results in another proper coloring.

By making the substitution
\[
	x_i
	=
	\left\{\begin{array}{ll}
	1&\text{if $i\in[t]$,}\\
	0&\text{if $i\notin[t]$,}
	\end{array}\right.
\]
all terms of $X(G;\bx)$ vanish except those corresponding to proper $[t]$-colorings of $G$, which become equal to $1$. Thus $X(G;\bx)$ reduces to $\chi(G;t)$ under this substitution, so the chromatic symmetric function is indeed a generalization of the chromatic polynomial.

To describe the analogue of Lemma~\ref{IM(S;[t])} in this context, we must introduce a particular basis for the ring of symmetric functions, the \emph{power sum symmetric functions}. All bases of the ring of symmetric functions can be indexed by integer partitions, and these are no exception. Suppose that $\la=(\la_1,\dots,\la_\ell)$ is a partition of the positive integer $n$, meaning that the $\la_i$ are positive integers called {\em parts} satisfying $\la_1+\cdots+\la_\ell=n$ and $\la_1\ge\cdots\ge\la_\ell$. We define $p_\la=p_\la(\bx)$ as the product
\[
	p_\la = p_{\la_1}\, p_{\la_2}\cdots p_{\la_\ell}
\]
where, for a positive integer $k$,
\[
	p_k = x_1^k + x_2^k + x_3^k + \cdots.
\]
For example,
\[
	p_{(3,1)}= p_3 p_1 = (x_1^3 + x_2^3 + x_3^3 +\cdots)(x_1 + x_2 + x_3 + \cdots).
\]

Finally, for a spanning subgraph $S\subseteq E$ of the graph $G=(V,E)$, we let $\la(S)$ denote the integer partition defined by listing the number of vertices in each component of $S$ in weakly decreasing order, so $\lambda(S)$ is a partition of $|V|$ into $c(S)$ parts.
For example, the spanning subgraph shown on the right of Figure~\ref{gc}  has $\la(S)=(3,1)$.
We then have the following analogue of Lemma~\ref{IM(S;[t])}.

\begin{lem}
\label{IM(S;x)}
For every graph $G=(V,E)$ and spanning subgraph $S\subseteq E$, we have
\[
	\sum_{\ka} \bx^\ka = p_{\lambda(S)},
\]
where the sum is over all $\mathbb{P}$-colorings of $G$ that are monochromatic on the components of $S$.
\end{lem}
\begin{proof}
The symmetric function for monochromatic colorings of a component with $k$ vertices is $p_k$. Since each component of $S$ may be colored independently, the product of the $p_k$ for all parts $k$ of $\la$ accounts for every coloring that is monochromatic on the components of $S$.
\end{proof}

We could have established Lemma~\ref{IM(S;x)} first and then derived Lemma~\ref{IM(S;[t])} from it by substitution as previously described. Indeed, throughout this article we present results in pairs like this where the first is a specialization of the second.
However, instead of establishing the first from the second by substitution, we establish the second by adapting the bijection used to prove the first.
These adaptations are quite minor.  The only differences in the symmetric function case are that the set of colorings is infinite, and that we have a weight function that associates a monomial with each coloring.
The cancellation we desire, which results in the equality
\[
	\sum_{s\in\cS} \sgn(s)\wt(s)
	=
	\sum_{s\in\Fix\io} \wt(s),
\]
still holds so long as the number of elements in $\cS$ of a given weight is finite and $\io$ is \emph{weight-preserving}, meaning that
$\wt(s)=\wt(\io(s))$ for all $s\in\cS$.

With this preparation complete, we are ready to construct our bijective proofs. In Section~\ref{sec-subgraphs}, we show how $\chi(G;t)$ and $X(G;\bx)$ can both be expressed as sums over all spanning subgraphs of $G$. Then in Section~\ref{sec-NBC}, we show (using a restriction of the bijection from Section~\ref{sec-subgraphs}) how they can both be expressed as sums over only those spanning subgraphs without broken circuits. Finally, in Sections~\ref{sec-acyclic} and \ref{sec-acyclic2} we give bijective proofs of Stanley's results linking $\chi(G;t)$ and $X(G;\bx)$ to acyclic orientations of graphs. It should be noted that these two sections are inspired by the work of Blass and Sagan~\cite{bs:bpt}. However, our approach differs from theirs in several respects.
First, we treat both the chromatic polynomial and the chromatic symmetric function while they only considered the former (the latter not having been defined when \cite{bs:bpt} was written).
Second, we give an explicit definition of the inverse of the bijection used in Section~\ref{sec-acyclic}, while they only described one direction and then showed it was one-to-one and onto.
Finally, we establish Stanley's theorem on acyclic orientations and compatible colorings for all negative values of $t$ in Section~\ref{sec-acyclic2}, whereas Blass and Sagan only considered the case of $t=-1$.

\section{Arbitrary spanning subgraphs.}
\label{sec-subgraphs}

Both $\chi(G;t)$ and $X(G;\bx)$ can be expressed as sums over the spanning subgraphs of $G$. In the case of $\chi(G;t)$, this expression is given by the following result of Birkhoff, from which it follows immediately that the chromatic polynomial is indeed a polynomial.

\begin{thm}
[Birkhoff~\cite{bir:dfn}]
\label{PAllS}
For every graph $G=(V,E)$ and every positive integer $t$, we have
\[
	\chi(G;t)
	=
	\sum_{S\subseteq E} (-1)^{|S|} t^{c(S)}.
\]
\end{thm}

To give a bijective proof of this result, we start by giving a combinatorial interpretation to the right-hand side. If $S\subseteq E$ is a spanning subgraph of $G$, then Lemma~\ref{IM(S;[t])} shows the number of $[t]$-colorings of $V$ that are monochromatic on the components of $S$ is $t^{c(S)}$. Thus the right-hand side of Theorem~\ref{PAllS} can be written as
\[
	\sum_{(S,\,\ka)\in\cS} (-1)^{|S|},
\]
where $\cS$ is the set of pairs defined by
\[
	\cS
	=
	\{(S,\ka):\text{$S\subseteq E$ and $\ka:V\ra[t]$ is monochromatic on the components of $S$}\}.
\]
At this point the reader is welcome to derive Theorem~\ref{PAllS} by inclusion--exclusion (it won't be difficult), but to practice for our later proofs, we employ a sign-reversing involution $\io$.

For the rest of this proof---in fact, for the rest of the article---it is necessary to fix a total ordering on the edges of $G$. With this order fixed, we describe the edges of $G$ as being \emph{first}, \emph{last}, \emph{earlier}, or \emph{later}, referring to their positions in this order.
We also define a sign function on the pairs in $\cS$ by
\[
	\sgn(S,\ka)= (-1)^{|S|},
\]
so
\[
	\sum_{S\subseteq E} (-1)^{|S|} t^{c(S)}
	=
	\sum_{(S,\,\ka)\in\cS} \sgn(S,\ka).
\]

We may now define $\io$, which is the star of the rest of this section and the next.
If $\ka$ is a proper coloring, we define $\io(S,\ka)=(S,\ka)$. As remarked earlier, this case can only occur if $S=\emptyset$, because otherwise $S$ must have at least one nontrivial component and thus the proper coloring $\ka$ cannot be monochromatic on its components.

Now suppose that $\ka$ has at least one monochromatic edge. We take $e$ to be the last such edge (in the ordering on the edges of $G$ fixed above) and define
\[
	\io(S,\ka)=(S\symmdiff e,\ka),
\]
where $\symmdiff$ is the symmetric difference operator, removing $e$ from $S$ if it is present and adding it to $S$ otherwise.

Our proof of Theorem~\ref{PAllS} is completed with the following proposition, which implies that
\[
	\sum_{S\subseteq E} (-1)^{|S|} t^{c(S)}
	=
	\sum_{(S,\,\ka)\in\cS} \sgn(S,\ka)
	=
	|\Fix\io|
	=
	\chi(G;t).
\]

\begin{prop}
\label{prop-io-sign-rev}
The mapping $\io$ is a sign-reversing involution on $\cS$ with
\[
	\Fix\io=\{(\emptyset,\ka)\in\cS \mid \text{$\ka$ is a proper coloring of $G$}\}.
\]
\end{prop}
\begin{proof}
First we verify that $\io$ maps $\cS$ to $\cS$. This is clear if $\ka$ is a proper coloring, so we must show that if $(S,\ka)\in\cS$ where $\ka$ is not proper, then $(S\symmdiff e,\ka)\in\cS$, i.e., $\kappa$ is still monochromatic on the components of $S\symmdiff e$. If $S\symmdiff e=S-e$, then this holds because each component of $S-e$ is contained in a component of $S$. If $S\symmdiff e=S\cup e$, then either $e$ connects two vertices of a component of $S$, or two components of $S$ were merged by the addition of the edge $e$. In the former case, $\kappa$ is clearly still monochromatic on components.  In the latter case, since $\ka$ is monochromatic on $e$, $\ka$ must have been the same color on both components that were merged. Thus in either case, $\ka$ is monochromatic on the components of $S\symmdiff e$, so $\io(S,\ka)\in\cS$ for all $(S,\ka)\in\cS$.

The fixed points of $\io$ are precisely the pairs $(\emptyset,\ka)$ where $\ka$ is a proper $[t]$-coloring of $S$. Moreover, these fixed points have positive sign, as desired. On the other elements of $\cS$, $\io$ is sign-reversing because $|S\symmdiff e|=|S|\pm 1$. Finally, $\io$ is an involution because $\ka$, and thus the definition of the edge $e$, does not change when passing from $(S,\ka)$ to $\io(S,\ka)$.
\end{proof}


The proof of the symmetric function analogue requires only minor modifications.

\begin{thm}[Stanley~{\cite[Theorem~2.5]{sta:sfg}}]
\label{XAllS}
For every graph $G=(V,E)$, we have
\[
	X(G;\bx) = \sum_{S\subseteq E} (-1)^{|S|}p_{\la(S)}.
\]
\end{thm}
\begin{proof}
Let $\cS$ denote the set of pairs $(S,\ka)$ where $S\subseteq E$ and $\ka$ is a $\mathbb{P}$-coloring of $G$ that is monochromatic on the components of $S$. Since $\io$ does not depend on the range of $\ka$, its definition extends to pairs where $\ka$ has range $\mathbb{P}$ without modification, as do both the statement and the proof of Proposition~\ref{prop-io-sign-rev}.

By defining the weight of a pair $(S,\ka)\in\cS$ as
\[
	\wt(S,\ka)=\bx^\ka
\]
and appealing to Lemma~\ref{IM(S;x)}, we see that
\[
	\sum_{S\subseteq E} (-1)^{|S|}p_{\la(S)}
	=
	\sum_{(S,\,\ka)\in\cS} \sgn(S,\ka) \wt(S,\ka).
\]
Since $\wt(S,\ka)$ depends only on $\ka$, which does not change in passing from $(S,\ka)$ to $\io(S,\ka)$, it follows that $\io$ is weight-preserving. Because $\io$ is (still) sign-reversing, every term in the sum on the right-hand side of the equation above cancels except those terms corresponding to pairs $(\emptyset,\ka)\in\cS$ where $\ka$ is a proper $\mathbb{P}$-coloring of $G$, and thus
\[
	\sum_{S\subseteq E} (-1)^{|S|}p_{\la(S)}
	=
	\sum_{(S,\,\ka)\in\cS} \sgn(S,\ka) \wt(S,\ka)
	=
	\sum_{(S,\,\ka)\in\Fix\io} \wt(S,\ka)
	=
	X(G;\bx),
\]
proving the theorem.
\end{proof}

\section{NBC spanning subgraphs.}
\label{sec-NBC}

In 1932, Whitney~\cite{whi:lem} showed that many of the terms in Theorem~\ref{PAllS} cancel with each other. The terms of the summation that remain after this pruning are precisely those that correspond to spanning subgraphs without broken circuits. In this section we show how this cancellation can be seen from our bijective proof of Theorem~\ref{PAllS}, and we also establish Stanley's symmetric function analogue of Whitney's result.

First we must define broken circuits. A \emph{walk} from the vertex $v_0$ to the vertex $v_k$ in the graph $G=(V,E)$ is an alternating sequence $v_0$, $e_1$, $v_1$, $e_2$, $v_2$, $\dots$, $e_k$, $v_k$ of vertices and edges such that for all $i\in[k]$ we have $e_i=v_{i-1}v_i\in E$.
This walk is a \emph{path} if it does not repeat any vertices or edges, except possibly the first and the last vertices. Note that we allow both walks and paths to be edgeless.
A walk or a path is \emph{closed} if $v_0=v_k$, that is, if its first and last vertices are the same. A \emph{cycle}, also known as a circuit, is a closed path with more than one vertex. Just as we conflate spanning subgraphs with their edge sets, we also conflate paths and cycles with their edge sets.
For example, the graph $G$ in Figure~\ref{bc} has a unique cycle which we denote as $C=\{e_1,e_2,e_3\}$.
Graphs not containing cycles can be referred to as \emph{acyclic}, but are more commonly called \emph{forests}.

Given a fixed total ordering on the edges $E$ of $G$, a \emph{broken circuit} is a subset $B\subseteq E$ of the form
\[
	B=C-\max C
\]
where $C$ is a cycle and $\max C$ is the last edge of $C$. For example, if the edges of the graph in Figure~\ref{bc} are ordered $e_1<e_2<e_3<e_4$, then the cycle $C=\{e_1,e_2,e_3\}$ corresponds to the broken circuit $B=\{e_1,e_2\}$. A spanning subgraph $S\subseteq E$ is \emph{NBC} (short for \emph{no broken circuits}) if it does not contain any broken circuits. It is frequently helpful to note that NBC spanning subgraphs are necessarily forests; since they don't contain broken circuits, they certainly don't contain cycles (as every cycle contains its own broken circuit). The rightmost picture in Figure~\ref{bc} represents an NBC spanning subgraph (together with colors that the reader may ignore for now).

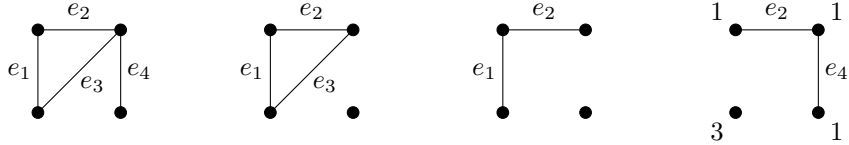
\begin{figure}
\begin{center}
\begin{footnotesize}
\begin{tikzpicture}[scale=1.1, baseline=(current bounding box.center)]
	\draw (0,0)--(1,1)--(0,1)--(0,0) (1,0)--(1,1);
  \draw [fill=black] (0,1) circle (2pt) node [above left, circle, inner sep=4pt] {\textcolor{white}{$1$}};
	\draw [fill=black] (1,1) circle (2pt) node [above right, circle, inner sep=4pt] {\textcolor{white}{$1$}};
	\draw [fill=black] (0,0) circle (2pt) node [below left, circle, inner sep=4pt] {\textcolor{white}{$3$}};
	\draw [fill=black] (1,0) circle (2pt) node [below right, circle, inner sep=4pt] {\textcolor{white}{$1$}};
	\node [left, circle, inner sep=1pt] at (0,0.5) {$e_1$};
	\node [above, circle, inner sep=1pt] at (0.5,1) {$e_2$};
	\node [below right, circle, inner sep=1pt] at (0.5,0.5) {$e_3$};
	\node [right, circle, inner sep=1pt] at (1,0.5) {$e_4$};
\end{tikzpicture}
\qquad
\begin{tikzpicture}[scale=1.1, baseline=(current bounding box.center)]
	\draw (0,0)--(1,1)--(0,1)--(0,0);
  \draw [fill=black] (0,1) circle (2pt) node [above left, circle, inner sep=4pt] {\textcolor{white}{$1$}};
	\draw [fill=black] (1,1) circle (2pt) node [above right, circle, inner sep=4pt] {\textcolor{white}{$1$}};
	\draw [fill=black] (0,0) circle (2pt) node [below left, circle, inner sep=4pt] {\textcolor{white}{$3$}};
	\draw [fill=black] (1,0) circle (2pt) node [below right, circle, inner sep=4pt] {\textcolor{white}{$1$}};
	\node [left, circle, inner sep=1pt] at (0,0.5) {$e_1$};
	\node [above, circle, inner sep=1pt] at (0.5,1) {$e_2$};
	\node [below right, circle, inner sep=1pt] at (0.5,0.5) {$e_3$};
\end{tikzpicture}
\qquad
\begin{tikzpicture}[scale=1.1, baseline=(current bounding box.center)]
	\draw (1,1)--(0,1)--(0,0);
  \draw [fill=black] (0,1) circle (2pt) node [above left, circle, inner sep=4pt] {\textcolor{white}{$1$}};
	\draw [fill=black] (1,1) circle (2pt) node [above right, circle, inner sep=4pt] {\textcolor{white}{$1$}};
	\draw [fill=black] (0,0) circle (2pt) node [below left, circle, inner sep=4pt] {\textcolor{white}{$3$}};
	\draw [fill=black] (1,0) circle (2pt) node [below right, circle, inner sep=4pt] {\textcolor{white}{$1$}};
	\node [left, circle, inner sep=1pt] at (0,0.5) {$e_1$};
	\node [above, circle, inner sep=1pt] at (0.5,1) {$e_2$};
\end{tikzpicture}
\qquad
\begin{tikzpicture}[scale=1.1, baseline=(current bounding box.center)]
	\draw (1,1)--(0,1) (1,0)--(1,1);
	\draw [fill=black] (0,1) circle (2pt) node [above left, circle, inner sep=4pt] {$1$};
	\draw [fill=black] (1,1) circle (2pt) node [above right, circle, inner sep=4pt] {$1$};
	\draw [fill=black] (0,0) circle (2pt) node [below left, circle, inner sep=4pt] {$3$};
	\draw [fill=black] (1,0) circle (2pt) node [below right, circle, inner sep=4pt] {$1$};
	\node [above, circle, inner sep=1pt] at (0.5,1) {$e_2$};
	\node [right, circle, inner sep=1pt] at (1,0.5) {$e_4$};
\end{tikzpicture}
\caption{A graph with ordered edges $e_1<e_2<e_3<e_4$, a cycle in this graph, the corresponding broken circuit, and an NBC spanning subgraph (with colors).}
\label{bc}
\end{footnotesize}
\end{center}
\end{figure}

Consider a particular pair $(S,\ka)$ from the set $\cS$ in the proof of Theorem~\ref{PAllS}, so $S$ is an arbitrary spanning subgraph of $G$ and $\ka$ is monochromatic on the components of $S$. Suppose that $S$ contains a broken circuit, say $B=C-\max C$. Since the vertices of $C$ lie in the same component of $S$ (because $B\subseteq S$), it follows that $\ka$ is monochromatic on the vertices of $C$, and thus in particular on the edge $\max C$. Therefore, if $S$ contains the broken circuit $B$, then we have $\io(S,\ka)=(S\symmdiff e,\ka)$ for an edge $e$ that is either equal to $\max C$ or occurs after it in the total ordering on edges of $G$. It follows that if the broken circuit $B$ is contained in $S$, then it is also contained in $S\symmdiff e$. We record this fact below.

\begin{prop}
\label{prop-fpf-sri}
The mapping $\io$ restricts to a fixed-point-free sign-reversing involution on the set of pairs $(S,\ka)$ where $S$ is a spanning subgraph containing a broken circuit and $\ka$ is monochromatic on the components of $S$.
\end{prop}

This result shows that the contributions of all such pairs cancel in the summation in Theorem~\ref{PAllS}, leaving us with only the NBC spanning subgraphs. This establishes Whitney's no broken circuits theorem.

\begin{thm}[Whitney~{\cite[Section~7]{whi:lem}}]
\label{WhiThm}
For every graph $G=(V,E)$, every total ordering of its edges, and every positive integer $t$, we have
\[
	\pushQED{\qed} 
	\chi(G;t)
	=
	\sum_{\substack{S\subseteq E,\\\textup{$S$ is NBC}\rule{0pt}{7pt}}}
		(-1)^{|S|} t^{c(S)}.
	\qedhere
	\popQED
\]
\end{thm}

The symmetric function analogue of the no broken circuit theorem follows  immediately from the observation, made in the proof of Lemma~\ref{XAllS}, that the quantity $\wt (S,\ka)=\bx^\ka$ is preserved by $\io$ because $\ka$ does not change.

\begin{thm}[Stanley~{\cite[Theorem~2.9]{sta:sfg}}]
\label{XNBC}
For every graph $G=(V,E)$ and every total ordering of its edges, we have
\[
	\pushQED{\qed} 
	X(G;\bx)
	=
	\sum_{\substack{S\subseteq E,\\\textup{$S$ is NBC}\rule{0pt}{7pt}}}
		(-1)^{|S|} p_{\lambda(S)}.
	\qedhere
	\popQED
\]
\end{thm}

\section{Acyclic orientations.}
\label{sec-acyclic}

Thus far we have only defined $\chi(G;t)$ for positive integers $t$, where Theorem~\ref{PAllS} shows that it is equal to a polynomial. For the rest of the article, we take this polynomial as the \emph{definition} of $\chi(G;t)$. With this change, we may evaluate $\chi(G;t)$ at negative values of $t$. Our first such evaluation, below, is a consequence of Whitney's no broken circuits theorem.
 
\begin{cor}
\label{WhiCor}
For every graph $G=(V,E)$, every total ordering of its edges, and every positive integer $t$, we have
\[
	\chi(G;-t)
	=
	(-1)^{|V|}\sum_{\substack{S\subseteq E,\\\textup{$S$ is NBC}\rule{0pt}{7pt}}}
		t^{c(S)}.
\]
\end{cor}
\begin{proof}
Theorem~\ref{WhiThm} shows us that
\[
	\chi(G;-t)
	=
	\sum_{\substack{S\subseteq E,\\\textup{$S$ is NBC}\rule{0pt}{7pt}}}
		(-1)^{|S|+c(S)} t^{c(S)}.
\]
If $S$ is an NBC spanning subgraph of $G$, then since it must be a forest, it has $|V|-|S|$ components. The result follows immediately.
\end{proof}

Stanley has established a beautiful interpretation of the value of $\chi(G;t)$ for negative integers $t$. Indeed, this relationship is one of the most striking examples of the phenomenon he later termed \emph{combinatorial reciprocity}~\cite{sta:crt}. In this section we give a bijective proof of this result in the special case of $t=-1$, following in the footsteps of Blass and Sagan~\cite{bs:bpt}. In the next section we demonstrate how this special case implies the theorem for all negative values of $t$, and also establish the analogous theorem for the chromatic symmetric function.

\begin{figure}
\begin{center}
\begin{footnotesize}
\begin{tikzpicture}[scale=1.1, baseline=(current bounding box.center)]
	\draw (0,0)--(1,1)--(0,1)--(0,0) (1,0)--(1,1);
  \draw [fill=black] (0,1) circle (2pt) node [above left, circle, inner sep=4pt] {\textcolor{white}{$3$}};
	\draw [fill=black] (1,1) circle (2pt) node [above right, circle, inner sep=4pt] {\textcolor{white}{$3$}};
	\draw [fill=black] (0,0) circle (2pt) node [below left, circle, inner sep=4pt] {\textcolor{white}{$1$}};
	\draw [fill=black] (1,0) circle (2pt) node [below right, circle, inner sep=4pt] {\textcolor{white}{$4$}};
\end{tikzpicture}
\qquad
\begin{tikzpicture}[scale=1.1, baseline=(current bounding box.center)]
	\draw[->-] (0,0)--(1,1);
	\draw[->-] (0,0)--(0,1);
	\draw[->-] (1,1)--(0,1);
	\draw[->-] (1,1)--(1,0);
	\draw (0,0)--(1,1)--(0,1)--(0,0) (1,0)--(1,1);
  \draw [fill=black] (0,1) circle (2pt) node [above left, circle, inner sep=4pt] {\textcolor{white}{$3$}};
	\draw [fill=black] (1,1) circle (2pt) node [above right, circle, inner sep=4pt] {\textcolor{white}{$3$}};
	\draw [fill=black] (0,0) circle (2pt) node [below left, circle, inner sep=4pt] {\textcolor{white}{$1$}};
	\draw [fill=black] (1,0) circle (2pt) node [below right, circle, inner sep=4pt] {\textcolor{white}{$4$}};
\end{tikzpicture}
\qquad
\begin{tikzpicture}[scale=1.1, baseline=(current bounding box.center)]
	\draw[->-] (0,0)--(1,1);
	\draw[->-] (0,0)--(0,1);
	\draw[->-] (1,1)--(0,1);
	\draw[->-] (1,1)--(1,0);
	\draw (0,0)--(1,1)--(0,1)--(0,0) (1,0)--(1,1);
	\draw [fill=black] (0,1) circle (2pt) node [above left, circle, inner sep=4pt] {$3$};
	\draw [fill=black] (1,1) circle (2pt) node [above right, circle, inner sep=4pt] {$3$};
	\draw [fill=black] (0,0) circle (2pt) node [below left, circle, inner sep=4pt] {$1$};
	\draw [fill=black] (1,0) circle (2pt) node [below right, circle, inner sep=4pt] {$4$};
\end{tikzpicture}
\caption{A graph, an acyclic orientation, and a compatible coloring (which can be ignored until Section~\ref{sec-acyclic2}).}
\label{gao}
\end{footnotesize}
\end{center}
\end{figure}
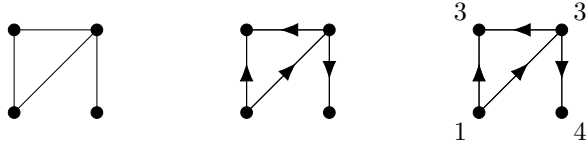

Stanley's interpretation of $\chi(G;t)$ for negative $t$ involves orientations of $G$, and so we must introduce oriented (or directed) edges, which we call \emph{arcs}. Given a pair of distinct vertices $u$ and $v$, there are two possible arcs between them, which we denote by $\ora{uv}=\ola{vu}$ and $\ola{uv}=\ora{vu}$. An \emph{orientation} of the (undirected) graph $G=(V,E)$ is obtained by replacing each edge $uv\in E$ by one of the arcs $\ora{uv}$ or $\ora{vu}$. Each such orientation is then an \emph{oriented graph}, which we denote by $O=(V,A)$, where $A$ is the set of arcs. Figure~\ref{gao} shows a graph on the left and one of its orientations in the center.

The bijection employed in this section proceeds edge-by-edge, so in the intermediate stages our graphs have both edges and arcs. It is therefore convenient to set our proof in the context of mixed graphs, whose study dates to a 1966 paper of Harary and Palmer~\cite{har:enu}. A \emph{mixed graph} is a triple $M=(V,E,A)$ where $V$ is a set of vertices, $E$ is a set of edges, and $A$ is a set of arcs. To eliminate the possibility of confusion, we briefly adapt our definitions of walks, paths, and cycles to this context.

A \emph{walk} from $v_0$ to $v_k$ in the mixed graph $M=(V,E,A)$ is an alternating sequence $v_0$, $c_1$, $v_1$, $c_2$, $v_2$, $\dots$, $c_k$, $v_k$ of vertices and edges/arcs such that for all $i\in[k]$, either $c_i=v_{i-1}v_i\in E$ or $c_i=\ora{v_{i-1}v_i}\in A$. We say that this walk \emph{traverses} each of these edges and arcs. We further call this walk a \emph{path} if it does not repeat any vertices, edges, or arcs, except possibly the first and last vertices.
For an example of this definition, the sequence $u$, $uv$, $v$, $vu$, $u$ is not considered to be a path because it repeats the edge $uv$, while the sequence $u$, $\ora{uv}$, $v$, $vu$, $u$ is considered to be a path.

We say that a walk is \emph{closed} if its first and last vertices are the same, and we call a closed path a \emph{cycle} if it has more than one vertex. For example, the sequence $u$, $\ora{uv}$, $v$, $vu$, $u$ of the previous paragraph is a cycle. Finally, a mixed graph is \emph{acyclic} if it does not contain a cycle. For example, the mixed graph (which also happens to be an oriented graph) in the center of Figure~\ref{gao} is acyclic.


One of the basic lemmas of graph theory is that if a graph contains a walk between two vertices, then it also contains a path between them (to prove this, simply shorten the walk each time it repeats a vertex). For mixed graphs, we have the following result whose proof is similar in spirit.

\begin{prop}
\label{prop-closed-walk}
If a mixed graph contains a closed walk that traverses at least one arc, then it also contains a cycle.
\end{prop}
\begin{proof}
Suppose that the proposition is not true, so there is an acyclic mixed graph $M$ that contains a closed walk traversing at least one arc. Let $W$ denote the shortest such walk in $M$, and label its vertices and edges/arcs as $v_0$, $c_1$, $v_1$, $c_2$, $v_2$, $\dots$, $c_k$, $v_k$. Since $W$ is closed, we have $v_0=v_k$, and since $W$ is not a cycle itself, we must also have $v_i=v_j$ for indices $i<j$ satisfying $\{i,j\}\neq\{0,k\}$.
Thus $W$ contains two shorter closed walks:
\[
	v_i, c_{i+1}, v_{i+1}, \dots, c_j, v_j
\]
and
\[
	v_j, c_{j+1}, v_{j+1}, \dots, c_k, v_k, c_1, v_1, \dots, c_i, v_i.
\]
These walks together traverse all edges and arcs traversed by $W$, so at least one of them must traverse an arc. This, however, contradicts the minimality of $W$, completing the proof.
\end{proof}

Stanley presented the following result as a corollary of Theorem~\ref{ComThm}. However, we appeal to Proposition~\ref{ComPro} in our proof of Theorem~\ref{ComThm}.

\begin{prop}
[Stanley~{\cite[Corollary~1.3]{sta:aog}}]
\label{ComPro}
For every graph $G=(V,E)$,
\[
	(-1)^{|V|}\chi(G;-1) =\text {the number of acyclic orientations of $G$}.
\]
\end{prop}
\begin{proof}
Fix a total ordering $e_1<e_2<\cdots<e_m$ of the edges of $G$. Corollary~\ref{WhiCor} then shows that
\[
	\chi(G;-1)
	=
	(-1)^{|V|}\sum_{\substack{S\subseteq E,\\\textup{$S$ is NBC}\rule{0pt}{7pt}}}
		1,
\]
or in words, that $(-1)^{|V|}\chi(G;-1)$ is equal to the number of NBC spanning subgraphs of $G$.

It therefore suffices to exhibit a bijection between the NBC spanning subgraphs of $G$ and its acyclic orientations. Define $\cM_i$ to be the set of acyclic mixed graphs comprised of
\begin{itemize}
\item an NBC subset of the edges $\{e_1,\dots,e_i\}$ and
\item an orientation of the edges $\{e_{i+1},\dots,e_m\}$.
\end{itemize}
Recall that an NBC subset of edges is automatically acyclic.
Thus $\cM_0$ and $\cM_m$ consist of the acyclic orientations and NBC spanning subgraphs of $G$, respectively. We establish the proposition by constructing bijections $\phi_i:\cM_{i-1}\to\cM_i$ for all indices $1\le i\le m$.

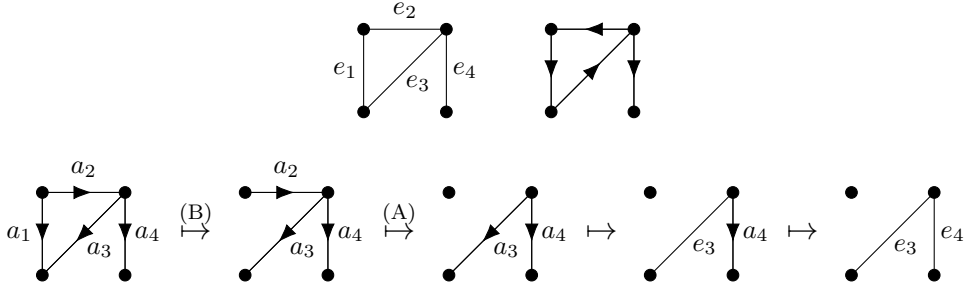
\begin{figure}
\begin{center}
\begin{footnotesize}
	\begin{tikzpicture}[scale=1.1]
	\draw (0,0)--(1,1)--(0,1)--(0,0) (1,0)--(1,1);
	\draw [fill=black] (0,1) circle (2pt);
	\draw [fill=black] (1,1) circle (2pt);
	\draw [fill=black] (0,0) circle (2pt);
	\draw [fill=black] (1,0) circle (2pt);
	\node [left, circle, inner sep=1pt] at (0,0.5) {$e_1$};
	\node [above, circle, inner sep=1pt] at (0.5,1) {$e_2$};
	\node [below right, circle, inner sep=1pt] at (0.5,0.5) {$e_3$};
	\node [right, circle, inner sep=1pt] at (1,0.5) {$e_4$};
	\end{tikzpicture}
\qquad
	\begin{tikzpicture}[scale=1.1]
	\draw [->-] (0,0)--(1,1);
	\draw [->-] (0,1)--(0,0);
	\draw [->-] (1,1)--(0,1);
	\draw [->-] (1,1)--(1,0);
	\draw (1,1)--(0,0)--(0,1)--(1,1)--(1,0);
	\draw [fill=black] (0,1) circle (2pt);
	\draw [fill=black] (1,1) circle (2pt);
	\draw [fill=black] (0,0) circle (2pt);
	\draw [fill=black] (1,0) circle (2pt);
	\end{tikzpicture}
\\[10pt]
	\begin{tikzpicture}[scale=1.1]
	\draw [->-] (1,1)--(0,0);
	\draw [->-] (0,1)--(0,0);
	\draw [->-] (0,1)--(1,1);
	\draw [->-] (1,1)--(1,0);
	\draw (1,1)--(0,0)--(0,1)--(1,1)--(1,0);
	\draw [fill=black] (0,1) circle (2pt);
	\draw [fill=black] (1,1) circle (2pt);
	\draw [fill=black] (0,0) circle (2pt);
	\draw [fill=black] (1,0) circle (2pt);
	\node [left, circle, inner sep=2pt] at (0,0.5) {$a_1$};
	\node [above, circle, inner sep=2pt] at (0.5,1) {$a_2$};
	\node [below right, circle, inner sep=2pt] at (0.5,0.5) {$a_3$};
	\node [right, circle, inner sep=2pt] at (1,0.5) {$a_4$};
	\end{tikzpicture}
\raisebox{16pt}{\normalsize $\stackrel{\text{(B)}}{\mapsto}$}\quad
	\begin{tikzpicture}[scale=1.1]
	\draw [->-] (1,1)--(0,0);
	\draw [->-] (0,1)--(1,1);
	\draw [->-] (1,1)--(1,0);
	\draw (1,1)--(0,0) (0,1)--(1,1)--(1,0);
	\draw [fill=black] (0,1) circle (2pt);
	\draw [fill=black] (1,1) circle (2pt);
	\draw [fill=black] (0,0) circle (2pt);
	\draw [fill=black] (1,0) circle (2pt);
	\node [above, circle, inner sep=2pt] at (0.5,1) {$a_2$};
	\node [below right, circle, inner sep=2pt] at (0.5,0.5) {$a_3$};
	\node [right, circle, inner sep=2pt] at (1,0.5) {$a_4$};
	\end{tikzpicture}
\raisebox{16pt}{\normalsize $\stackrel{\text{(A)}}{\mapsto}$}\quad
	\begin{tikzpicture}[scale=1.1]
	\draw [->-] (1,1)--(0,0);
	\draw [->-] (1,1)--(1,0);
	\draw (1,1)--(0,0) (1,1)--(1,0);
	\draw [fill=black] (0,1) circle (2pt);
	\draw [fill=black] (1,1) circle (2pt);
	\draw [fill=black] (0,0) circle (2pt);
	\draw [fill=black] (1,0) circle (2pt);
	\node [below right, circle, inner sep=2pt] at (0.5,0.5) {$a_3$};
	\node [right, circle, inner sep=2pt] at (1,0.5) {$a_4$};
	\end{tikzpicture}
\raisebox{16pt}{\normalsize $\mapsto$}\quad
	\begin{tikzpicture}[scale=1.1]
	\draw [->-] (1,1)--(1,0);
	\draw (1,1)--(0,0) (1,1)--(1,0);
	\draw [fill=black] (0,1) circle (2pt);
	\draw [fill=black] (1,1) circle (2pt);
	\draw [fill=black] (0,0) circle (2pt);
	\draw [fill=black] (1,0) circle (2pt);
	\node [below right, circle, inner sep=1pt] at (0.5,0.5) {$e_3$};
	\node [right, circle, inner sep=2pt] at (1,0.5) {$a_4$};
	\end{tikzpicture}
\raisebox{16pt}{\normalsize $\mapsto$}\quad
	\begin{tikzpicture}[scale=1.1]
	\draw (0,0)--(1,1)--(1,0);
	\draw [fill=black] (0,1) circle (2pt);
	\draw [fill=black] (1,1) circle (2pt);
	\draw [fill=black] (0,0) circle (2pt);
	\draw [fill=black] (1,0) circle (2pt);
	\node [below right, circle, inner sep=1pt] at (0.5,0.5) {$e_3$};
	\node [right, circle, inner sep=1pt] at (1,0.5) {$e_4$};
	\end{tikzpicture}
\caption{Applying the mappings $\phi_1$, $\phi_2$, $\phi_3$, and $\phi_4$ to an acyclic orientation of the graph shown in the top left, where the normal orientations of the edges are as shown in the top right.}
\label{psiEx}
\end{footnotesize}
\end{center}
\end{figure}

To define these bijections, we need to distinguish between the two orientations of each edge of $G$. How we make this distinction is immaterial, so we simply fix one orientation of each edge $e_i$ as \emph{normal}, denoted by $\ora{e_i}$, and  call the opposite orientation $\ola{e_i}$ \emph{abnormal}.
We may now define $\phi_i:\cM_{i-1}\to\cM_i$.
Take $M\in\cM_{i-1}$ and suppose that the edge $e_i$ of $G$ appears as the arc $a_i$ in $M$. The mixed graph $\phi_i(M)$ is obtained by either \emph{unorienting} $a_i$ (replacing the arc $a_i$ with the edge $e_i$) or \emph{removing} it (deleting the arc $a_i$). The rules by which we make this choice are as follows.  We unorient $a_i$ if both
\begin{enumerate}
	\item[(A)] $a_i$ is the normal orientation of $e_i$ and
	\item[(B)] unorienting $a_i$ does not create a cycle.
\end{enumerate}
Otherwise, we remove $a_i$.
An example of applying the mappings $\phi_i$ is shown in Figure~\ref{psiEx}, where the label above each arrow indicates which of (A) or (B) was violated if the arc was removed (there is no such label if the arc satisfied both (A) and (B) and thus was unoriented instead).

We must check that $\phi_i$ produces members of $\cM_i$. It follows immediately that $M'=\phi_i(M)$ is acyclic because rule (B) prevents the creation of a cycle. 
However, how does $\phi_i$ avoid creating broken circuits?
The answer is that were $M'$ to contain a broken circuit, then---since $M\in\cM_{i-1}$ and thus does not itself contain a broken circuit---the last two edges in the corresponding cycle of $G$ must be $e_i$ and $e_j$ for some $j>i$. However, $M$ contains some orientation of $e_j$, so unorienting $a_i$ would result in a cycle. Therefore rule (B) ensures that $a_i$ is removed, and thus $M'$ cannot contain a broken circuit. This verifies that $M'\in\cM_i$.

Our proof is completed by constructing the inverse of $\phi_i$, which we denote by $\psi_i:\cM_i\to\cM_{i-1}$. Given a mixed graph $M'\in\cM_i$, the mixed graph $M=\psi_i(M')$ is obtained by adding one of the orientations of edge $e_i$ to $M'$ (and removing $e_i$ if it is present in $M'$).
We give $e_i$ the \emph{ab}normal orientation if both
\begin{enumerate}
\item[($\text{A}'$)] $e_i$ is not an edge of $M'$ and
\item[($\text{B}'$)] adding $\ola{e_i}$ to $M'$ does not create a cycle.
\end{enumerate}
Otherwise, we give $e_i$ the normal orientation.

Again we must check that $M\in\cM_{i-1}$. Orienting an edge cannot cause $M$ to contain an NBC set, so we need only check that $M$ is acyclic.
This follows from ($\text{B}'$) if $e_i$ was oriented abnormally. If $e_i$ was oriented normally, then either ($\text{A}'$) or ($\text{B}'$) was violated.
If it was ($\text{A}'$) that was violated, then $e_i\in M'$, and $M$ is acyclic because orienting an edge that is already present cannot create a cycle.

This leaves us to consider the case where $e_i$ was oriented normally because condition ($\text{B}'$) was violated. Suppose $\ola{e_i}=\ola{uv}$, so the fact that condition ($\text{B}'$) was violated means that $M'$ contains a path from $u$ to $v$, which we denote by $P$. If $P$ were to consist entirely of edges, then $P\subseteq\{e_1,e_2,\dots,e_{i-1}\}$ because $M'\in\cM_i$. However, in this case $e_i$ is the greatest edge of the cycle $P\cup e_i$ of $G$, so $P$ is a broken circuit. As this contradicts our assumption that $M'\in\cM_i$, it follows that $P$ must traverse at least one arc. Therefore, if $M'$ were to also contain a path from $v$ to $u$, then this path together with $P$ would be a closed walk that traversed at least one arc. Proposition~\ref{prop-closed-walk} would then imply that $M'$ contains a cycle, a contradiction. We may therefore conclude that $M'$ does not contain a path from $v$ to $u$. Thus orienting $e_i$ as $\ora{e_i}=\ora{uv}$, as it is oriented in $M=\psi_i(M')$, does not create a cycle.

It remains only to check that $\psi_i\circ\phi_i$ and $\phi_i\circ\psi_i$ are the identity mappings on $\cM_{i-1}$ and $\cM_i$, respectively. We show only one of these, as the other is similar.
Consider $\psi_i(\phi_i(M))$ where $M\in \cM_{i-1}$, and let $M'=\phi_i(M)$.
There are two cases, depending on whether $M'$ is obtained from $M$ by unorienting $e_i$ or by removing it.
In the first case $e_i$ was oriented normally in $M$ and $e_i$ is an edge of $M'$. Thus ($\text{A}'$) is violated and $\psi(M')$ is obtained by orienting $e_i$ normally, as desired.

For the second case, suppose that $e_i$ was removed, so either (A) or (B) was violated.
If (A) was violated, then $e_i$ was abnormally oriented in $M$ and we must show that $\psi_i(M')$ adds the arc back in that orientation, or in other words, that both ($\text{A}'$) and ($\text{B}'$) are satisfied.
We know that ($\text{A}'$) is satisfied because $e_i$ was removed from $M$ due to its abnormal orientation.
We also see that ($\text{B}'$) holds because orienting $e_i$ as $\ola{e_i}$ results in the mixed graph $M$ which was assumed to be acyclic.
For the final subcase, assume that (A) holds but (B) does not. This means that $e_i$ was oriented normally in $M$ but was removed because (B) was violated. Thus the cycle that violates (B) must be created by the possibility of traversing $e_i$ in its abnormal orientation since, by assumption, the normal orientation of $e_i$ is part of $M$ which is acyclic.
In this case ($\text{B}'$) is violated, and $e_i$ is added back to $M'$ in its normal orientation, as desired.
\end{proof}

\section{Multi-colored acyclic orientations.}
\label{sec-acyclic2}

Proposition~\ref{ComPro} is a consequence of a more general result in Stanley's 1973 paper~\cite{sta:aog}. Here, by adapting an idea essentially due to Vo~\cite[Section~3]{vo:graph-colorings:}, we show how this more general result follows bijectively from Proposition~\ref{ComPro}, and then establish the chromatic symmetric function analogue. Observe that it is not at all clear algebraically that Proposition~\ref{ComPro} implies the more general Theorem~\ref{ComThm}, but viewing the results combinatorially makes this implication possible.

Suppose that $O=(V,A)$ is an orientation of the graph $G=(V,E)$ and that $\ka$ is a coloring of $G$ (by $[t]$ or by $\mathbb{P}$).
We say that
$O$ and $\ka$ are \emph{compatible} if
\[
	\ora{uv}\in A
	\text{ implies }
	\ka(u)\le \ka(v),
\]
that is, if all the arcs of $O$ point in the direction of weakly increasing values of $\ka$. It is convenient to display such pairs by labeling the vertices of $O$ with their values under $\ka$. Using this convention, we see a compatible pair on the right in Figure~\ref{gao}.
We can now state the generalization of  Proposition~\ref{ComPro} to all negative values of $t$.

\begin{thm}[Stanley~{\cite[Theorem~1.2]{sta:aog}}]
\label{ComThm}
For every graph $G=(V,E)$ and every positive integer $t$ we have
\begin{multline*}
	(-1)^{|V|}\chi(G;-t)
	=\\
	|\{(O,\ka) \mid \text{$O$ is an acyclic orientation of $G$ and $\ka$  a compatible $[t]$-coloring}\}|.
\end{multline*}
\end{thm}
\begin{proof}
Fix a total ordering of the edges of $G$, so by Corollary~\ref{WhiCor} we have
\[
	(-1)^{|V|}\chi(G;-t)
	=
	\sum_{\substack{S\subseteq E,\\\textup{$S$ is NBC}\rule{0pt}{7pt}}}
		t^{c(S)}.
\]
Lemma~\ref{IM(S;[t])} shows that this sum counts pairs $(S,\ka)$ where $S$ is an NBC spanning subgraph of $G$ and $\ka$ is a $[t]$-coloring that is monochromatic on the components of $S$. We would like to construct a bijection between such pairs and pairs of the form $(O,\ka)$, where $O$ is an acyclic orientation of $G$ and $\ka$ is a compatible $[t]$-coloring.
To prove the theorem, we construct, for every $[t]$-coloring $\ka$ of $G$, a bijection $\Phi$ between the sets
\begin{eqnarray*}
	\cO_\ka&=&\{O:\text{$O$ is an acyclic orientation of $G$ compatible with $\ka$}\}\ \text{and}\\
	\cS_\ka&=&\{S:\text{$S\subseteq E$ is NBC and $\ka$ is monochromatic on the components of $S$}\}.
\end{eqnarray*}

Note that when $t=1$, there is only one coloring of $G$ (the constant coloring), and every acyclic orientation is compatible with this coloring. Thus in this case $\Phi$ must restrict to a bijection from acyclic orientations of $G$ to its NBC spanning subgraphs. Indeed, we constructed just such a bijection in the previous section, although any such bijection will do for our purposes here.

For each $i\in[t]$, define
\[
	V_i=\{v\in V\mid \ka(v)=i\},
\]
and let $E_i$ denote the set of (monochromatic) edges of $G$ between vertices in $V_i$. Note that $E_1\cup E_2\cup\cdots\cup E_t$ needn't contain all of the edges of $G$, but for every $S\in\cS_\ka$ we have
\[
	S\subseteq E_1\cup E_2\cup\cdots\cup E_t
\]
because $\ka$ is monochromatic on the components of $S$. Similarly, for each orientation $O\in\cO_\ka$, we know how the edges outside $E_1\cup E_2\cup\cdots\cup E_t$ are oriented because $O$ and $\ka$ are compatible.
Indeed, any edge $uv$ where $u\in V_i$, $v\in V_j$, and $i<j$ must be oriented as $\ora{uv}$.
Let $C$ denote the set of these arcs.

We may now define $\Phi$. Take some orientation $O\in\cO_\ka$, and for each $i\in[t]$, let $O_i$ denote the sub-orientation consisting of all arcs of $O$ between vertices of color $i$, so $O_i$ is an orientation of the spanning subgraph $E_i$. Because we are assuming Proposition~\ref{ComPro}, we know that for each $E_i$, there is a bijection from the acyclic orientations of $E_i$ to its NBC spanning subgraphs. Let this bijection be denoted by $\phi^{(i)}$.
Thus $\phi^{(i)}(O_i)$ is an NBC spanning subgraph of $E_i$ for every $i\in[t]$. We may now define $\Phi(O)$ simply by
\[
	\Phi(O)
	=
	\phi^{(1)}(O_1)\cup\phi^{(2)}(O_2)\cup\cdots\cup\phi^{(t)}(O_t).
\]

Note that the components of $\phi^{(i)}(O_i)$ are subsets of $E_i$. This implies that $\ka$ is monochromatic on the components of $\Phi(O)$. This fact also allows us to see that $\Phi(O)$ is an NBC spanning subgraph---if $\Phi(O)$ were to contain a broken circuit, then that broken circuit would lie in a single component, and thus it would be contained in $\phi^{(i)}(O_i)$ for some $i$, but this is a contradiction. It follows that $\Phi(O)\in\cS_\ka$.

Next we construct the inverse of $\Phi$, which we call $\Psi$.
To define it, consider an NBC spanning subgraph $S\in\cS_\ka$. Because $\ka$ is monochromatic on the components of $S$, we know that $S\subseteq E_1\cup E_2\cup\cdots\cup E_t$. For each $i\in[t]$, let $S_i=S\cap E_i$, so $S_i$ is an NBC spanning subgraph of $E_i$. 
For each $i\in[t]$, let $\psi^{(i)}$ denote the inverse of the bijection $\phi^{(i)}$. We then define $\Psi(S)$ by
\[
\Psi(S)
	=
	C\cup\psi^{(1)}(S_1)\cup\psi^{(2)}(S_2)\cup\cdots\cup\psi^{(t)}(S_t)
\]
where $C$ is the set of arcs defined earlier (whose orientations are determined by compatibility).

We need to make sure that $\Psi(S)$ is acyclic. Suppose, toward a contradiction, that this orientation contains a cycle. The vertices visited by this cycle cannot all have the same color since each $\psi^{(i)}(S_i)$ is acyclic. Thus we may label the vertices visited by the cycle as $v_0$, $v_1$, $v_2$, $\dots$, $v_k=v_0$ where $\ka(v_0)<\ka(v_1)$. However, the fact that $\ka$ is compatible with this orientation implies that
\[
	\ka(v_0)<\ka(v_1)\le \ka(v_2)\le\dots\le\ka(v_k)=\ka(v_0).
\]
This contradiction---that $\ka(v_0)<\ka(v_0)$---establishes that the orientation is acyclic, showing that $\Psi(S)\in\cO_\ka$. It follows that $\Phi$ is a bijection with inverse $\Psi$ since the analogous statement is true of the $\phi^{(i)}$ and $\psi^{(i)}$. This completes the proof.
\end{proof}

To state the symmetric function analogue of Theorem~\ref{ComThm}, we first need to introduce a well-known involution $\om$ on the ring of symmetric functions. When $\la=(\la_1,\dots,\la_\ell)$ is a partition of $n$ into $\ell$ parts, we define
\[
	\om(p_\la)
	=
	(-1)^{n-\ell} p_\la.
\]
We then extend $\om$ to the entire ring linearly.


\begin{thm}[Stanley~{\cite[Theorem~4.2]{sta:sfg}}]
For every graph $G$, we have
\[
	\om(X(G;\bx)) = \sum_{(O,\,\ka)} \bx^\ka,
\]
where the sum is over all pairs $(O,\ka)$ where $O$ is an acyclic orientation of $G$ that is compatible with the $\mathbb{P}$-coloring $\ka$.
\end{thm}
\begin{proof}
From Theorem~\ref{XNBC} we have that
\[
	X(G;\bx)
	=
	\sum_{\substack{S\subseteq E,\\\textup{$S$ is NBC}\rule{0pt}{7pt}}}
		(-1)^{|S|} p_{\lambda(S)}.
\]
All spanning subgraphs $S$ appearing in this sum are forests (because they are NBC), so $\lambda(S)$ is a partition of $|V|$ into $|V|-|S|$ parts. Thus, applying $\om$ to the equation above yields
\[
	\om(X(G;\bx))
	=
	\sum_{\substack{S\subseteq E,\\\textup{$S$ is NBC}\rule{0pt}{7pt}}}
		(-1)^{|S|} (-1)^{|V|-(|V|-|S|)} p_{\lambda(S)}
	=
	\sum_{\substack{S\subseteq E,\\\textup{$S$ is NBC}\rule{0pt}{7pt}}}
		p_{\lambda(S)}.
\]
By Lemma~\ref{IM(S;x)}, the right-hand side of this equation is the weight generating function for pairs $(S,\ka)$ where $S$ is an NBC set and $\ka$ is monochromatic on the components of $S$. Since the bijection $\Psi$ of the previous proof preserves colorings, it also preserves the quantity $\bx^\ka$, so this bijection shows that
\[
	\sum_{\substack{S\subseteq E,\\\textup{$S$ is NBC}\rule{0pt}{7pt}}}
		p_{\lambda(S)}
	=
	\sum_{(O,\,\ka)} \bx^\ka,
\]
finishing the proof of the theorem.
\end{proof}


%
%
%
%
%

\section{Concluding remarks.}

Even after more than a century of study, several questions about the chromatic polynomial remain open. To name a particularly natural and tantalizing one, which was raised by Read in his 1968 survey~\cite[Section~9]{rea:icp} and then again (independently, it seems) by Wilf~\cite{wil:wpc} eight years later: which polynomials are chromatic polynomials? Of course, the analogous question for chromatic symmetric functions is just as natural.

Although the problem of characterizing the chromatic polynomials remains unsolved, we at least have a description of their coefficients via Theorem~\ref{WhiThm}. In particular, since every NBC set $S$ is a forest, if $|V|=n$ and $|S|=k$, then $c(S)=n-k$. Thus Whitney's theorem shows that
\[
	\chi(G;t)=\sum_{k\ge0} (-1)^k a_k t^{n-k}
\]
where $a_k$ is the number of NBC spanning subgraphs of $G$ with $k$ edges.  One can now ask if the coefficient sequence $a_0, a_1,\ldots, a_n$ has any interesting properties.  Call such a sequence of real numbers {\em log-concave} if
\[
	a_k^2 \ge a_{k-1} a_{k+1}
\]
for all $0<k<n$.  Log-concave sequences abound in combinatorics, algebra, and geometry, as can be seen from the survey articles of Stanley~\cite{sta:lus}, Brenti~\cite{bre:lus}, and Br\"and\'en~\cite{bra:ulr}.  In 2012, Huh~\cite{huh:mnp} stunned the combinatorial community by using deep results from algebrac geometry to prove that the coefficient sequence of $\chi(G;t)$ is always log-concave.  A weaker property had been conjectured by Read~\cite{rea:icp} in 1968.

As we have seen, mixed graphs arise naturally in the study of the chromatic polynomial. Beck, Bogart, and Pham~\cite{beck:enumeration-of-:} have extended the definition of chromatic polynomials to this context. Specifically, they define the \emph{strong chromatic polynomial} of the mixed graph $G=(V,E,A)$ to be the number of $[t]$-colorings $\ka$ satisfying
\begin{enumerate}
	\item[(a)] $\ka(u)\neq\ka(v)$ if $uv\in E$ and
	\item[(b)] $\ka(u)<\ka(v)$ if $\ora{uv}\in A$.
\end{enumerate}
They show that this is a polynomial in $t$. They also prove an analogue of Proposition~\ref{ComPro} for these polynomials, phrased in the language of hyperplane arrangements.
(Given any graph $G$ with vertex set $[n]$, one associates a set of hyperplanes where $ij\in E$ corresponds to the hyperplane $x_i=x_j$. It is then easy to see that the acyclic orientations of $G$ correspond to the regions obtained by removing the hyperplanes from $\mathbb{R}^n$.) Beck, Blado, Crawford, Jean-Louis, and Young~\cite{beck:on-weak-chromat:} later introduced the \emph{weak chromatic polynomial} of a mixed graph, which counts colorings satisfying the weak version of the inequality in (b). They establish a mixed graph analogue of Theorem~\ref{ComThm} for these polynomials.


\begin{figure}
\begin{center}
	\begin{tikzpicture}[scale=1.1, baseline=(current bounding box.center)]
	\draw (0,0)--(30:1)--(-30:1)--cycle;
	\draw (0,0)--(150:1)--(210:1)--cycle;
	\draw [fill=black] (0,0) circle (2pt);
	\draw [fill=black] (30:1) circle (2pt);
	\draw [fill=black] (-30:1) circle (2pt);
	\draw [fill=black] (150:1) circle (2pt);
	\draw [fill=black] (210:1) circle (2pt);
	\end{tikzpicture}
\quad\quad
	\begin{tikzpicture}[scale=1.1, baseline=(current bounding box.center)]
	\draw ({-sqrt(3)/2},0)--(0,{1/2})--({sqrt(3)/2},0)--(0,{-1/2})--cycle;
	\draw (0,{1/2})--(0,{-1/2});
	\draw ({sqrt(3)/2},0)--({1+sqrt(3)/2},0);
	\draw [fill=black] ({1+sqrt(3)/2},0) circle (2pt);
	\draw [fill=black] ({-sqrt(3)/2},0) circle (2pt);
	\draw [fill=black] (0,{1/2}) circle (2pt);
	\draw [fill=black] ({sqrt(3)/2},0) circle (2pt);
	\draw [fill=black] (0,{-1/2}) circle (2pt);
	\end{tikzpicture}
	\begin{tikzpicture}[scale=1.1, baseline=(current bounding box.center)]
	\draw (90:{1/sqrt(3)})--(210:{1/sqrt(3)})--(330:{1/sqrt(3)})--cycle;
	\draw (90:{1/sqrt(3)})--(90:{1+1/sqrt(3)});
	\draw (210:{1/sqrt(3)})--(210:{1+1/sqrt(3)});
	\draw (330:{1/sqrt(3)})--(330:{1+1/sqrt(3)});
	\draw [fill=black] (90:{1/sqrt(3)}) circle (2pt);
	\draw [fill=black] (210:{1/sqrt(3)}) circle (2pt);
	\draw [fill=black] (330:{1/sqrt(3)}) circle (2pt);
	\draw [fill=black] (90:{1+1/sqrt(3)}) circle (2pt);
	\draw [fill=black] (210:{1+1/sqrt(3)}) circle (2pt);
	\draw [fill=black] (330:{1+1/sqrt(3)}) circle (2pt);
	\end{tikzpicture}
\quad\quad
	\begin{tikzpicture}[scale=1.1, baseline=(current bounding box.center)]
	\draw (0,0)--(-60:1)--(240:1)--cycle;
	\draw (0,0)--(120:1);
	\draw (0,0)--(60:2);
	\draw [fill=black] (0,0) circle (2pt);
	\draw [fill=black] (-60:1) circle (2pt);
	\draw [fill=black] (240:1) circle (2pt);
	\draw [fill=black] (120:1) circle (2pt);
	\draw [fill=black] (60:1) circle (2pt);
	\draw [fill=black] (60:2) circle (2pt);
	\end{tikzpicture}
\end{center}
\caption{From left to right, the butterfly (or bowtie), the kite, the net, and $X_{169}$.}
\label{fig-Xequal}
\end{figure}
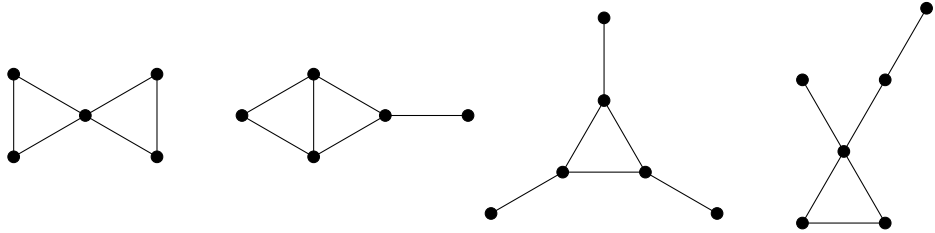

Finally, we would be remiss if we did not mention a question of Stanley's
that remains open despite significant interest. First, in terms of the definitions we have established, a \emph{tree} is a forest with a single connected component. For any tree $T$ on $n$ vertices, it is easy to see that $\chi(T;t)=t(t-1)^{n-1}$, so we say that the chromatic polynomial does not distinguish trees. However, Stanley~\cite[p.~170]{sta:sfg} asked whether the opposite is true for the chromatic symmetric function. Specifically, if $T_1$ and $T_2$ are nonisomorphic trees, is it true that $X(T_1;\bx)\neq X(T_2;\bx)$? Stanley repeated this question as Exercise~7.47.b in Volume 2 of \emph{Enumerative Combinatorics}~\cite{stanley:enumerative-com:2}. It has since become known as the tree conjecture. A series of researchers have verified the tree conjecture by computer for small trees; at present the record belongs to Heil and Ji~\cite{heil:on-an-algorithm:}, who have shown that the conjecture holds for all trees on 29 or fewer vertices. If true, the tree conjecture would imply that the chromatic symmetric function is a \emph{complete tree invariant}.

Note that the chromatic symmetric function does \emph{not} distinguish all graphs. The example Stanley~\cite{sta:sfg} gives is that the chromatic symmetric functions of the butterfly and the kite are equal (see Figure~\ref{fig-Xequal}). Orellana and Scott~\cite{os:gec} have observed that the chromatic symmetric function also does not distinguish unicyclic graphs (those with a single cycle). The example they present is the net and the graph that we refer to---following the \emph{Information System on Graph Classes and Their Inclusions}~\cite{ISGCI} for lack of a better name---as $X_{169}$ (also shown in Figure~\ref{fig-Xequal}).

%
%
%
%


\end{document}